\documentclass[a4paper,12pt]{article}
\usepackage{complexity}
\usepackage{lineno}
\usepackage{fullpage}
\usepackage{hyperref}
\usepackage{setspace}
\usepackage{tikz}
\usepackage{xcolor}
\usepackage{enumerate}
\usepackage[T1]{fontenc}
\usepackage{lmodern}
\usepackage{graphicx}
\usepackage{textgreek}
\usepackage{indentfirst}
\usepackage{pstricks,pspicture}
\usepackage{float}
\usepackage{amsthm,amsmath,amssymb}
\usepackage{pst-plot}
\usepackage{epsfig}
\usepackage{setspace}
\usepackage{float}
\usepackage{rotating}
\usepackage[left=2cm,top=2cm, bottom=2cm,right=2cm]{geometry}
\def\slika #1{\begin{center} \epsffile{#1} \end{center}}

\newtheorem{theorem}{Theorem}
\newtheorem{observation}{Observation}
\newtheorem{definition}[theorem]{Definition}
\newtheorem{lemma}[theorem]{Lemma}

\newtheorem{proposition}[theorem]{Proposition}
\newtheorem{remark}[theorem]{Remark}

\title{Carath$\'{e}$odory Number and Exchange Number in $\Delta$-convexity} \date{}
\author{
Bijo S. Anand$^1$,
Arun Anil$^2$,
Manoj Changat$^2$,
Prasanth G. Narasimha-Shenoi$^{3,4}$\\ and
Sabeer S. Ramla$^5$
}

\begin{document}
\maketitle
%\linenumbers
\begin{center}
{\small

$^1$Department of Mathematics, Sree Narayana College, Punalur\\ Kollam, Kerala, India - 691305.\\ e-mail: bijos$\_$anand@yahoo.com\\

$^{2}$Department of Futures Studies, University of Kerala, \\Thiruvananthapuram, Kerala, India - 695581.\\
e-mails: arunanil93@gmail.com, mchangat@keralauniversity.ac.in,\\

$^3$Department of Mathematics, Government College Chittur,\\ Palakkad, Kerala, India - 678104.\\ e-mail: prasanthgns@gmail.com \\

$^4$Department of Collegiate Education, Government of Kerala, \\Thiruvananthapuram, Kerala, India - 695033.

$^5$Department of Mathematics, Thangal Kunju Musaliar Institute of Technology, \\ Musaliar Hills, Karuvelil P.O., Ezhukone, Kollam, Kerala, India - 691505.\\ e-mail: sabeersainmaths@gmail.com
}
\end{center}
\begin{abstract}
    Given a graph $G$, a set is $\Delta$ convex if there is no vertex $u\in V(G)\setminus S$
that forms a triangle with two vertices of $S$. The $\Delta$-convex hull of $S$ is the minimum $\Delta$-convex set containing $S$. This article is an attempt to discuss the Carath\'eodory number and exchange number on various graph families and standard graph products namely Cartesian, strong and lexicographic products of graphs.
\end{abstract}

\noindent Keywords: Carath\'eodory number, exchange number, graph products, block graphs.\\
\noindent AMS subject classification: 05C38, 05C76, 05C99, 52A01.
\section{Introduction}
Like many fundamental concepts in mathematics, convexity permeates various disciplines in mathematics like geometry, optimization, economics, and functional analysis.  For more see ~\cite{boyd2004convex,rockefellar1970convex}. Van de Vel \cite{van_de_vel}, developed a comprehensive axiomatic structure for convexity theory that later became a building stone for many researchers to study convexity and its parameters in a large way.  A convexity on a set is a family of subsets stable for intersection and nested unions.  
%A set $S$ of vertices in a graph is said to be convex if the path joining any two vertices of $S$ should also be in $S$.  

It can be seen that many notions of convexity have been introduced and studied in the cases of graphs also. The pioneers like Degreef, Doignon et al., Jamison, Reay, Sierksma, and van de Vel studied exchange properites of convex sets in axiomatic convexity, see \cite{degreef1982convex,doignon1981tverberg,jamison1974general,reay1965generalizations,sierksma1976relationships,van_de_vel}.  The invariant exchange number was introduced and studied by Sierksma in 1975 \cite{sierksma1975caratheodory}. Similarly, the exchange number and Carath\'eodory number when the geodesic convexity is considered on graphs have been studied in \cite{anand2020caratheodory}.  The relationship between the exchange number and classical convexity invariants like the Carath\'eodory, Helly and Radon numbers has been studied by the above authors. It is worth to mention the remark by van de
Vel \cite{van_de_vel} that the exchange number is the only invariant that exactly determines the dimension of convexity spaces and plays an essential role in proving the equivalence of classical Helly’s and Carath\'eodory’s theorems.\\
In continuation to these studies, various authors studied different type of convexities on graphs,  see \cite{dourado2016geodetic,farber1987bridged,pelayo2013geodesic} where the convexity under consideration is the geodesic convexity; see \cite{changat2004induced,duchet1988convex,morgana2002induced} when
 the monophonic convexity is considered, see \cite{centeno2011irreversible,centeno2013geodetic} for the case of $P_3$-convexity and see \cite{changat2005convexities,dourado2016complexity} when the triangle path convexity is under consideration. Among the studied problems are the counterparts of the classical parameters Carath\'eodory, Helly, and Radon numbers. These parameters have been studied in the monophonic convexity in \cite{duchet1987convexity}, in the triangle path convexity in \cite{changat1999triangle} and in the $P_3$-convexity in \cite{coelho2014caratheodory,dourado2013algorithmic}. For invariants and related structures see \cite{anand2015helly,sierksma1980convexity}.  For a more general source see \cite{van_de_vel}.  Another convexity that is studied very recently is the $\Delta$-convexity.  In \cite{anand2020convexity}, authors studied the $\Delta$-number and $\Delta$-convexity number and in particular, the exact value for the convexity number of block graphs with diameter $\leq 3$ and on graph products namely strong and lexicographic product of two graphs. In \cite{bsaamc}, authors prove that computing the $\Delta$-convex hull of a general graph is NP-complete and also discussed various polynomial type algorithms for some special graph families. 
 
% Several properties like exchange property is studied in various branches of mathematics like matrix theory, matroids, where in matroid theory this property ensures the sets namely independent sets have a structure similar to linearly independent sets in vectorspaces see \cite{whitney1992abstract}.  In grapth theory also the exchange property has been studied by various authors.
%\section{Preliminaries}
Throughout this article, we consider finite, simple, and undirected graphs $G=(V(G), E(G))$, where $V(G)$ is the vertex set and $E(G)$ is the edge set of the graph $G$.  If there is no confusion we use $V, E$ for vertex set and edge set of $G$. A graph $H$ is said to be a \emph{subgraph} of $G$ if $V(H)\subseteq V(G)$ and $E(H)\subseteq E(G)$. A subgraph $H$ of $G$ is an \emph{induced subgraph} of $G$ if for $u,v\in V(H)$
and $uv\in E(G)$ implies $uv\in E(H)$.  By a \textit{chordal} graph we mean a graph in which all cycles of four or more vertices have a chord.  A graph $G$ is said to be \textit{complete}, if every vertex is adjacent to every other vertex.  A complete graph on $n$ vertices is denoted by $K_n$.  A graph $G$ is said to be $H$-\textit{free} if $G$ does not contain $H$ as an induced subgraph.  For example, $G$ is said to be triangle-free, if $G$ does not contain $K_3$ as an induced subgraph.  A \textit{block} in a graph 
$G$ is a maximal connected subgraph of 
$G$ that does not have a cut vertex. A graph $G$ is said to be a \textit{block graph} if every block is complete.  A \textit{chain of blocks} is a sequence of blocks in a graph such that each consecutive pair of blocks shares a common cut vertex. For further reading, refer \cite{west2001introduction}.

Given a graph $G$ and a set $S \subseteq V(G)$, the {\em $\Delta$-interval of $S$}, $[S]_\Delta$, is the set formed by the vertices of $S$ and every $w \in V(G)$ such that $w$ forms a triangle with two vertices of $S$. Sometimes it will be useful to write $[u,v]_\Delta$ standing for $[\{u,v\}]_\Delta$. If $[S]_\Delta = S$, then $S$ is {\em $\Delta$-convex of $G$}. A finite \emph{convexity space} is a pair $(X,\mathcal{C})$ where $X$ is a finite set and $\mathcal{C}$ is a collection of subsets of $X$ such that; $\emptyset, X \in \mathcal{C}$ and $\mathcal{C}$ is closed under intersections.
The sets in $\mathcal{C}$ are considered the \textit{convex sets}. The \textit{convex hull}  of a subset $S$ of $X$ is the intersection of all convex sets containing $S$, denoted by $\langle S\rangle$. Equivalently, it is the smallest superset of $S$ in $\mathcal{C}$.

Let $S\subseteq X$ be a nonempty finite set.
The set $S$ is \textit{Carath\'{e}odory dependent} (or,\textit{C-dependent}) provided $\displaystyle\langle S \rangle\subseteq\bigcup_{a\in S}\langle S \setminus \{a\}\rangle ,$ and it is \textit{Carath\'{e}odory independent} (or, \textit{C-independent})  otherwise. That is, a subset $S$ of $X$ is said to be C-independent, if there is $\displaystyle p\in \langle S\rangle\setminus\bigcup_{a\in S}\langle S\setminus\{a\}\rangle$.
The set $S$ is called \textit{exchange dependent}(or, \textit{E-dependent}), provided for each $p \in S$, $\displaystyle\langle S \setminus \{p\}\rangle\subseteq\bigcup_{a\in S\setminus\{p\}} \langle S \setminus \{a\}\rangle$, and it is called \textit{exchange independent } (or, \textit{E-independent}) otherwise. That is, $S$  is said to be E-independent, if  $|S|=1$ or there is a $p\in S$ such that  $\displaystyle p'\in \langle S\setminus\{p\}\rangle\setminus\bigcup_{a\in S\setminus\{p\}}\langle S\setminus\{a\}\rangle$. In this case, we say that $p$ is a pivot of $S$.
The set $S$ is  \emph{Helly dependent} (or, \emph{H-dependent}) provided $\displaystyle \bigcap_{a\in S}\langle S\setminus\{a\}\rangle \neq \emptyset$, and it is \emph{Helly independent} (or, \emph{H-independent})  otherwise.

A convex structure $X$ gives rise to the following numbers $c(X),e(X), h(X) \in \{0,1,2,3,\ldots, \infty\}$ determined by the following prescription. For each $n$ with $0\leq n < \infty$,
$c(X)\leq n$ if and only if each finite set $S\subseteq X$ with $|S|>n$ is C-dependent;
$e(X)\leq n$ if and only if each finite set $S\subseteq X$ with $|S|>n$ is E-dependent;  $h(X)\leq n$ if and only if each finite set $S\subseteq X$ with $|S|>n$ is H-dependent;
where $c(X)$ is the \textit{Carath\'{e}odory number} of $X$,  $e(X)$ is the \textit{exchange number} of $X$ and $h(X)$ is the \textit{Helly number} of $X$.
That is, $c(X)$ is the maximum cardinality of a $C$-independent set,
$e(X)$ is the maximum cardinality of a $E$-independent set and $h(X)$ is the maximum cardinality of a $H$-independent set.

In $\Delta$-convexity, the parameters Carath\'eodory, exchange and Helly are respectively denoted as $c_\Delta(X)$, $e_\Delta(X)$ and $h_\Delta(X)$ or simply $c,e,h$ if there is no confusion.

A
property $\pi$ is \textit{hereditary} if every subset of a set with property $\pi$ also has property $\pi$. It is known that Radon and Helly
independence are hereditary properties while Carath\'eodory and exchange independence are not hereditary in general, see \cite{van_de_vel}.

For all three products of (simple) graphs, $G$ and $H$ the vertex set of the product is $V(G)\times V(H)$. Their edge sets are defined as follows. In the \emph{Cartesian product} $G\Box H$ two vertices are adjacent if they are adjacent in one coordinate and equal in the other. Two vertices $(g_1, h_1)$ and $(g_2, h_2)$ in the strong product of $G$ and $H$, $G\boxtimes H$ are adjacent  if one of the following holds: 
i) $g_1g_2\in E(G)$ and $h_1 = h_2$, 
ii) $g_1 = g_2$ and $h_1h_2\in E(H)$, or 
iii) $g_1 g_2\in E(G)$ and $h_1h_2\in E(H)$.

 Finally, two vertices $(g,h)$ and $(g^{\prime },h^{\prime })$ are adjacent in
\emph{lexicographic product} $G\circ H$ (also $G[H]$) if $gg^{\prime}\in E(G)$ or if $g=g^{\prime }$ and $hh^{\prime }\in E(H)$. For
$\ast \in \{\Box ,\boxtimes,\circ \}$ we call the product $G\ast H$ \emph{nontrivial}, if both $G$ and $H$ have at least two vertices.
For $h\in V(H)$, $g\in V(G)$, and $\ast \in \{\Box ,\boxtimes,\circ \}$, call $G^{h}=\{(g,h)\in G\ast H:\,g\in V(G)\}$ a $G$ \emph{layer} in $G\ast H$,
and call $^{g}H=\{(g,h)\in G\ast H:\,h\in V(H)\}$ an $H$ $\emph{layer}$ in $G\ast H$. Note that the subgraph of $G\ast H$ induced on $G^{h}$ is isomorphic to $G$ and the subgraph of $G\ast H$ induced on $^{g}H$ is isomorphic to $H$ for $\ast
\in \{\Box ,\boxtimes ,\circ \}$. The map
$p_{G}:V(G\ast H)\rightarrow V(G)$ defined with $p_{G}((g,h))=g$ is called a \emph{projection map onto} $G$ for $\ast \in \{\Box ,\boxtimes ,\circ \}$. Similarly, we can define the \emph{projection map onto} $H$.
Let $S\subset V(G)$. With $\left\langle S\right\rangle $ we denote the subgraph of $G$ induced by $S$. 

\begin{theorem}[Sierksma Inequalities \cite{van_de_vel}]\label{inequalities}
For all convex structures, $e - 1 \leq c \leq \max\{h, e - 1\}$, where $e$ is the exchange number, $h$ is the Helly number, and $c$ is the Carath\'eodory number.
\end{theorem}

We organize the rest of the sections as follows. In Section \ref{section-caratheodory}, we establish an upper bound for the Carath\'{e}odory number and determine its exact values for block graphs. In Section \ref{section-exchange}, we address the exchange number in $\Delta$-convexity by providing an upper bound for the exchange number and its exact values for block graphs. Finally in Section \ref{section-product}, we present sharp lower bound for exchange number and Carath\'{e}odory number for $\Delta$-convexity in Cartesian product of graphs and the precise values for the  strong and the lexicographic product of graphs.

%\textcolor{blue}{We have to define pivot of E-independent set}\\
%\textbf{\textcolor{red}{Why cann't we learn: Let $P$ be a subset of vertices of graph $G$. A \textit{Radon partition} of
%$P$ into two nonempty subsets $P_t$ and $Q_t$ is called $t$ - \textit{tolerant Radon
%partition}, if for any set $S \subseteq P$ with $|S| \leq t$, we have $<P_t  - S> \cap <Q_t  - %S>\neq \emptyset$. For $t$-tolerant Radon partition, each partition
%must have at least $t + 1$ vertices.
%Clearly, $t$ - tolerant Radon partitions of a set are $(t-1)$ - tolerant
%Radon partitions. But the converse need not be true.}}

%%%%%%%%%%%%%%%%%%%%%%%%%%%%%%%%%%%%%%%%%%%%%%%%%%%%%%%

	\section{ Carath\'{e}odory number}\label{section-caratheodory}
	
		In this section, first we establish an upper bound for the Carath\'{e}odory number in the context of $\Delta$-convexity and then we provide an example that achieves this bound. Furthermore, we demonstrate that for any natural number $n$, there exist a graph in which the Carath\'{e}odory number is equal to $n$. Finally, we determine the Carath\'{e}odory number for block graphs. We start the with following two observations which follow immediately from the definition of  $\Delta$-convexity.
		\begin{observation}
			If $G$ is a triangle-free graph, then the Carath\'eodory number $c_\Delta(G)=1$.
		\end{observation}

			\begin{observation}
				If $G$ is a complete graph $K_n$ for $n>2$, then $c_\Delta(K_n)=2$.
			\end{observation}

		\begin{theorem}\label{caratheodory_bound}
			If $G$ is a graph with $k$ triangles, then the Carath\'{e}odory number $c_\Delta(G) \leq k+1$.
		\end{theorem}
		\begin{proof}

			Let $S\subseteq V(G)$ such that $|S|>1$.
			
			If $S$ contains a vertex $v$ not lying on any triangle, then  $\langle S\rangle=\langle S\setminus \{v\}\rangle\cup \{v\}$. Hence, $\langle S\rangle= \langle S\setminus \{v\}\rangle\cup \langle S\setminus \{u\}\rangle$, where $u\neq v$. This implies that $S$ is a C-dependent set. That is, a C-independent set contains only the vertices of triangles in $G$.
			
			If $S$ contains all vertices $\{u,v,w\}$ of a triangle, then $\langle S\rangle=\langle S\setminus \{u\}\rangle=\langle S\setminus \{v\}\rangle=\langle S\setminus \{w\}\rangle$. This implies that $S$ is a C-dependent set. That is, a C-independent set does not contain all the vertices of a triangle in $G$.
			
			Now, we aim to prove that for any $S\subseteq V(G)$ with $|S|>k+1$, it is C-dependent.
			
			Assume that there exists a Carathéodory independent set $S'$ with  $|S'|\geq k+2$. Since there are only $k$ triangles in $G$, $S'$ contains two vertices from some triangles $T_1,T_2,T_3,\ldots,T_j$, where $j\geq 2$. Let $u_i,v_i \in S'$ and $u_i,v_i \in T_i$ for $i\in\{1,2,\ldots,j\}$. Then $\displaystyle  \langle S'\rangle=\bigcup_{i=1}^{j}\left(\langle S'\setminus \{u_i\}\rangle\cup \langle S'\setminus \{v_i\}\rangle\right)$ implies that \( S' \) is C-dependent. Thus, for any \( S\subseteq V(G) \) with \( |S|>k+1 \), it is C-dependent.
			Therefore, the Carathéodory number of \( G \), \( c_\Delta\leq k+1 \).
		\end{proof}

If $k = 1$, then $G$ contains only one triangle. In this case, the Carath\'{e}odory number, $ c_\Delta(G) = 2$.

			\begin{figure}[H]
                \centering
			\epsfxsize=15truecm \slika{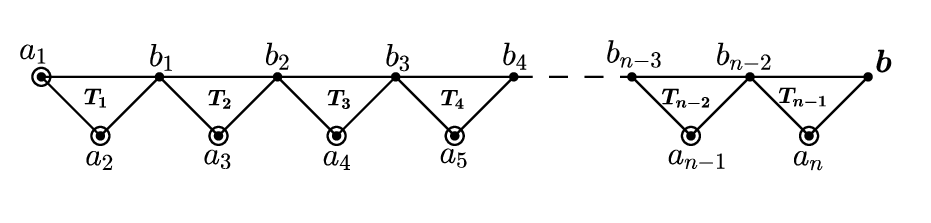}\vspace{-0.4in}
			\caption{Graph $G$ with with $c_\Delta(G)=e_\Delta(G)=n$.}
			\label{delta}
		\end{figure}
		\begin{proposition}
			For any natural number $n$, there exists a graph $G$ with $c_\Delta(G)=n$.
		\end{proposition}
		
		\begin{proof}
		For $n=1$, consider any triangle-free graph $G$. In this case, $c_\Delta(G)=1$.\\
		For $n=2$, consider any complete graph $G$. Here, $c_\Delta(G)=2$.\\
        For $n\geq 3$, consider the graph $G$ depicted in Figure \ref{delta}. Let $S=\{a_1,a_2,a_3,\ldots,a_n\}$. Then $\langle S\rangle=V(G)$, $\langle S\setminus\{a_1\}\rangle=\{a_2,a_3,a_4,\ldots, a_n\}$, $\langle S\setminus\{a_2\}\rangle=\{a_1,a_3,a_4,\ldots ,a_n\}$, $\langle S\setminus\{a_i\}\rangle=\{a_1,a_2,b_1,a_3,b_2,a_4,\ldots ,a_{i-1},b_{i-2},a_{i+1},a_{i+2},\ldots ,a_n\}$ for $i\in\{3,4,5,\ldots ,n-1\}$ and $\langle S\setminus\{a_n\}\rangle=\{a_1,a_2,b_1,a_3,b_2,a_4,\ldots,b_{n-3},a_{n-1},b_{n-2}\}$.\\ However, $\displaystyle b\notin\bigcup_{a_i\in S} \langle S\setminus\{a_i\}\rangle$. Therefore, $S$ is a Carathéodory independent set, and $|S|=n$. Since there are $n-1$ triangles in $G$, by Theorem \ref{caratheodory_bound}, $c_\Delta(G)=n$.
		\end{proof}
\begin{remark}
	There exists a graph $G$ with $k$ triangles that have Carathéodory number of $k+1$. For example, in Figure \ref{delta}, $G$ is a graph with $n-1$ triangles and the Carathéodory number is $n$.
\end{remark}

%\begin{lemma}\cite{bsaamc}\label{hull_properties}
	%If $\langle u,v \rangle= S$, then, for every vertex $x \in S$, there is a $y \in S \cap N(x)$ such that $\langle x,y \rangle = S$.
%\end{lemma}

In the next result, we present some properties of the Carath\'{e}odory  independent set of a general graph, which will be used in subsequent results.

\begin{lemma}\label{cara_properties}
If $S$ is a  Carath\'{e}odory independent set of $G$ and $|S|\geq2$, then $S$ has the following properties.
	\begin{enumerate}

            \item[(i)] the elements of $S$ are part of some triangles in $G$.
		\item[(ii)] the induced subgraph of $S$ contains at least one edge.
            \item[(iii)] for any $u,v\in S$, $\langle S\setminus\{u,v\}\rangle \cap \langle u,v  \rangle=\emptyset$.
  	
	\end{enumerate}
\end{lemma}
\begin{proof} 

   \begin{enumerate}

   \item[(i)] If $S$ contains a vertex $v$ not lying on any triangle, then $ \langle S\rangle=\langle S \setminus \{v\}\rangle\cup \{v\}$. Hence, $\langle S\rangle= \langle S \setminus  \{v\}\rangle\cup \langle S \setminus  \{u\}\rangle$, where $u\neq v$. This implies that $S$ is a Carath\'eodory dependent set.
		
  \item[(ii)] Let $S$ be a $C$-independent set. Assume that the induced subgraph of $S$ does not contain any edges; that is, $S$ contains pairwise non-adjacent vertices. Hence, $\langle S\rangle = S$, and for any $u\in S$, we have  $\langle S\setminus \{u\} \rangle = S\setminus \{u\}$. This implies that $S$ is a $C$-dependent set, a contradiction. Therefore, the induced subgraph of  $S$ must contain at least one edge.

  \item[(iii)] Let $u,v\in S$ and assume that $uv \in E(G)$. By property (ii), since $S$ contains only one edge, the induced subgraph of $S\setminus\{u,v\}$ will not contain an edge. Therefore, $\langle S\setminus\{u,v\}\rangle = S\setminus\{u,v\}$.
  Now, since $uv \in E(G)$ and by property (i) there exists at least one vertex $w\in V(G) \setminus S, w \neq u,v$ such that $w \in \langle u,v  \rangle$. Without loss of generality, we assume that $\langle S\setminus\{u,v\}\rangle \cap \langle u,v  \rangle \neq\emptyset$ and let $ w \in \langle S\setminus\{u,v\}\rangle \cap \langle u,v  \rangle $. But this is not possible, since $w\in V(G) \setminus S, w \neq u,v $ we have $ w \notin \langle S\setminus\{u,v\}\rangle = S\setminus\{u,v\}$.

  It remains to prove it for $u,v\in S$ where $uv \notin E(G)$. Since $uv \notin E(G)$ we have $\langle u,v  \rangle = \{u,v\}$. If $\langle S\setminus\{u,v\}\rangle \cap \langle u,v  \rangle \neq\emptyset$ implies that either $ u \in \langle S\setminus\{u,v\}\rangle$ or $ v \in \langle S\setminus\{u,v\}\rangle$. However, both conditions cannot be true because if $S$ a Carath\'{e}odory independent set then for any vertex $u\in S$, $u\notin \langle S\setminus\{u\}\rangle$.

   \end{enumerate}
\end{proof}

To prove that the Carath\'{e}odory number of a 2-connected chordal graph is two, we need the following result.
\begin{lemma}\cite{bsaamc}\label{hull_2_chordal}
	If $G$ is a 2-connected chordal graph, then every pair of adjacent vertices form a hull set of $G$.
\end{lemma}
%\qed
\begin{theorem}
 If $G$ is a 2-connected chordal graph, then $c_\Delta(G)=2$.   
\end{theorem}
\begin{proof}
	Let $ S=\{u,v\}\subseteq  V(G)$ with $uv\in E(G)$. By Lemma \ref{hull_2_chordal},  we have $\langle S \rangle = V(G)$. Since $\langle S \rangle \not\subseteq \langle S \setminus u \rangle  \bigcup \langle S \setminus v \rangle$,  $S$ is C-independent. Now, we need to prove that any subset $S'\subseteq V(G)$ with $|S'|\geq 3$ is C-dependent. There are two cases to consider: 
	\begin{description}
		\item[Case 1:] $S'$ consists of pairwise non-adjacent vertices. By Lemma \ref{cara_properties} (ii), $S'$ is C-dependent.
		\item[Case 2:] $S'$ contains at least two vertices, say $u, v$, such that $uv\in E(G)$. By Lemma \ref{hull_2_chordal},  $\langle u,v \rangle = V(G)$. Consequently, for any $w\in S'\setminus\{u,v\}$, we have $\langle S'\setminus \{w\}\rangle= V(G)$. This implies $S'$ is C-dependent.	
	\end{description}
Since every subset of size at least 3 is C-dependent, we conclude that $c_\Delta(G) = 2 $, completing the proof.
\end{proof}

We conclude this section by presenting the exact value for the Carath\'{e}odory number of  block graphs. For the following result, we define a block of a graph $G$ as a maximal 2-connected subgraph.

\begin{theorem}
Let $G$ be a block graph with $\ell$ number of blocks.
\begin{itemize}
\item[(i)]  If $G$ contains no blocks isomorphic to the complete graph $K_2$ and all blocks lie in a single chain, then $c_{\Delta}(G)=\ell+1$.
\item[(ii)]  If $G$ contains no blocks isomorphic to the complete graph $K_2$, and the blocks are on different chains with $k$ being the number of blocks in the longest chain, then $c_{\Delta}(G)=k+1$.
\item[(iii)]  If $G$ contains blocks isomorphic to the complete graph $K_2$ and $k$ is the maximum number of consecutive blocks which does not contains $K_2$ as a block among the chains of blocks in $G$, then $c_{\Delta}(G)=k+1$.
\end{itemize}
\end{theorem}	 

\begin{proof}
\begin{itemize}
\item[(i)] Let $G$ be a block graph with $\ell$ blocks, where no blocks are isomorphic to $K_2$ and all blocks lie in a single chain. Let $B_1, B_2, B_3, \ldots, B_{\ell}$ be the blocks of $G$ with $B_1$ and $B_{\ell}$ being the pendant blocks and let $S$ be a C-independent set of $G$. By Lemma~\ref{cara_properties}, $S$ contains two vertices from at least one block and not more than two vertices from a single block. Without loss of generality we may assume that $S$ contains two vertices from $B_1$ and one vertex from all other blocks. Let $S=\left\{u_1, u_1{ }^{\prime}, u_2, u_3, \ldots, u_{\ell}\right\}$, where $u_1, u_1{ }^{\prime} \in B_1, u_i \in B_i$ for $i\in\{2,3, \ldots, \ell\}$, and there are no three vertices in $S$ forming $K_3$ in $G$ and the vertices are not cut vertices of $G$. By construction, $\langle S\rangle=V(G)$. Also, $\langle S\setminus \{u_1\}\rangle=S \setminus\{ u_1\}$,
$\langle S \setminus\{u_1^{\prime}\}\rangle=S \setminus \{u_1^{\prime}\}, \langle S\setminus \{u_i\}\rangle=V(B_1)\cup \cdots \cup V(B_{i-1}) \cup\{u_{i+1}, u_{i+2}, \ldots, u_{\ell}\}$ for $i\in\{2,3, \ldots, \ell\}$. This implies that there exists at least one vertex $v \in V\left(B_{\ell}\right)$ such that $v \in\langle S\rangle \setminus \bigcup_{a \in S}\langle S \setminus\{a\}\rangle$. Thus, $S$ is a C-independent set of size $\ell+1$ elements.

Now, we need to prove that there is no C-independent set of size greater than $\ell+1$. Suppose there is a C-independent set $S^{\prime}$ with $\left|S^{\prime}\right| \geq \ell+2$. By Lemma \ref{cara_properties}, $S^{\prime}$ contains two vertices from at least one block and not more than two vertices from a single block. Consider the possible cases:
\begin{description}

	\item[Case 1:] $S^{\prime}$ contains at least one vertex from each block.
    
    Since $\left|S^{\prime}\right| \geq \ell+2, S^{\prime}$ contains at least two vertices each from two blocks, say $B_i$ and $B_j$ with $i, j \in\{1,2,3, \ldots, \ell\}$. Let $u, u^{\prime}, v, v^{\prime} \in S^{\prime}$, where $u, u^{\prime} \in V(B_i)$ and $v, v^{\prime} \in V(B_j)$. Then $\langle S \setminus\{u\}\rangle=\langle S \setminus\{v\}\rangle=V(G)$, making $S^{\prime}$ a C-dependent set.
     \item[Case 2:] $S^{\prime}$ does not contain vertices from some blocks.
     
     That is, $S^{\prime}$ contains two vertices from at least three blocks. There exists a sequence of blocks $B_i, B_{i+1}, \ldots, B_j$ such that $S^{\prime}$ contains at least one vertex from each of these blocks and two vertices from at least two blocks, say $B_x$ and $B_y$ with $x, y \in\{i, i+1, \ldots, j\}$. Let $u, u^{\prime}, v, v^{\prime} \in S^{\prime}$, where $u, u^{\prime} \in V(B_x)$ and $v, v^{\prime} \in V(B_y)$. Then $\langle S \setminus\{u\}\rangle=\langle S \setminus\{v\}\rangle=\left\langle S^{\prime}\right\rangle$, making $S^{\prime}$ a C-dependent set.
\end{description}
Thus, there is no C-independent set of size greater than $\ell+1$. Hence, the Carathéodory number $c_{\Delta}(G)=\ell+1$.

\item[(ii)] Let $G$ be a block graph containing no blocks isomorphic to the complete graph $K_2$, and let the blocks be on different chains with $k$ being the number of blocks in the longest chain. Then, from (i), there is a C-independent set of size $k+1$ in the longest chain, denoted as $S$. Now, we need to prove that there is no C-independent set of size greater than $k+1$. Suppose there is a C-independent set $S^{\prime}$ with $\left|S^{\prime}\right| \geq k+2$. By Lemma \ref{cara_properties}, $S^{\prime}$ contains two vertices from at least one block and not more than two vertices from a single block. From (i), it is not possible to get a C-independent set of size greater than $k+1$ from one single chain of blocks. So, it is clear that $S^{\prime}$ contains vertices from at least two distinct chains of blocks.

Consider the possible cases:
			\begin{description}
				\item[Case 1:] $S'$ contains only elements from two distinct chains of blocks.\\
				Let the chains of blocks be $C_1$ and $C_2$  with $u\in V(C_1)\cap S'$,$w \in V(C_2)\cap S'$ and $u\notin V(C_2)\cap S'$,$w \notin V(C_1)\cap S'$. Since $|S'|\geq k+2$, in $\langle {S'}\setminus\{u\}\rangle$ contain all elements of $\langle S'\rangle$ except some elements of $\langle S'\rangle \cap V(C_1)$. Similarly, $\langle {S'}\setminus\{w\}\rangle$ contain all elements of $\langle S'\rangle$ except some elements of $\langle S'\rangle \cap V(C_2)$. Then $ \displaystyle \langle {S'}\setminus\{u\}\rangle \cup \langle {S'}\setminus\{w\}\rangle = \langle S'\rangle$. Thus $S'$ is a C-dependent set.
				\item[Case 2:] $S'$  contains elements from more than two  distinct longest chains of blocks.\\
				Let $C_1$, $C_2$,\ldots,$C_n$ be the least possible number of distinct longest chains of blocks with $u_i\in V(C_i)\cap S'$ and $u_i$ not in any other chains $C_j$, $j\neq i$ for $i\in\{1,2,\ldots,n\}$. In $\langle {S'}\setminus\{u_i\}\rangle$ contains all elements of $\langle S'\rangle$ except some  elements of $\langle S'\rangle \cap V(C_i)$ for $i\in\{1,2,\ldots,n\}$.  Since $|S'|\geq k+2$,  $\displaystyle  \bigcap_{m=1} ^{n}\left( \langle S' \rangle \cap V(C_i)\right)\subseteq \bigcup_{m=1}^{n} \langle S'\setminus\{u_i\}\rangle$. That is, $\displaystyle\langle S' \rangle =\bigcup_{m=1}^{n} \langle S'\setminus\{u_i\}\rangle$. This implies $S'$ is an C-dependent set.
			\end{description}
			Thus, there is no C-independent set of size greater than $k+1$. Hence, the Carathéodory number  $c_{\Delta}(G) = k+1$.

\item[(iii)]  Let $G$ be a block graph containing blocks isomorphic to the complete graph $K_2$, and let $k$ be the maximum number of consecutive blocks that does not contain $K_2$ as a block among the chains of blocks in $G$. From (i), there is a C-independent set of size $k+1$ from these consecutive blocks.

Now, we need to prove that there is no C-independent set of size greater than $k+1$. Suppose there is a C-independent set $S^{\prime}$ with $\left|S^{\prime}\right| \geq k+2$. By Lemma \ref{cara_properties}, $S^{\prime}$ contains two vertices from at least one block and no more than two vertices from a single block. Since it is not possible to obtain an C-independent set of size $k+2$ from consecutive blocks of a single chain, $S'$ contains at most $k+1$ elements from consecutive blocks of a single chain.

  So $S'$ contains elements from non-consecutive blocks. From Lemma \ref{cara_properties} (ii), there exist at least two vertices say, $u,v\in S'$ with $uv\in E(G)$, let  $C$ be the consecutive block of chain containing $u$ and $v$. Let $w$ be in $S'$ but not in $C$. Then, both $\langle S\setminus\{u\}\rangle$ and $\langle S\setminus\{v\}\rangle$ contain all elements of $\langle S'\rangle$ except some elements of $V(C)\cap \langle S'\rangle$. Also, $\langle S\setminus\{w\}\rangle$ contain  all the elements in $V(C)\cap \langle S'\rangle$. Therefore, $\langle {S'}\setminus\{u\}\rangle \cup \langle {S'}\setminus\{v\}\rangle\cup \langle {S'}\setminus\{w\}\rangle = \langle S'\rangle$. This implies, $S'$ is an C-dependent set.
			
			Thus, there is no C-independent set of size greater than $k+1$. Hence, the Carathéodory number $c_{\Delta}(G) = k+1$.
\end{itemize}
\end{proof}
 %%%%%%%%%%%%%%%%%%%%%%%%%%%%%%%%%%%%%%%%%%
\section{Exchange number}\label{section-exchange}
	In this section, we deal with the exchange number in the context of $\Delta$-convexity. %We establish an upper bound for the exchange number for general graphs. we demonstrate that for any natural number $n$, there exists a graph in which the exchange number is equal to $n$. Furthermore, for block graphs and 2-connected chordal graphs, we determine the exact values for the exchange number.
The following lemma discuss some necessary conditions for a set to be an E-independent set.
	
\begin{lemma}\label{adjacency_exchange_indpt}
			If $S$ is an E-independent set of $G$ and $|S|\geq3$, then the following holds.
			\begin{itemize}
				\item[(i)]induced subgraph of $S$ contains at least one edge of $G$.
				\item[(ii)] no three vertices of $S$ form a $K_3$ in $G$.
				\item[(iii)] there is at most one vertex in $S$ that is not part of a $K_3$ in $G$.
			\end{itemize}
		\end{lemma}
		\begin{proof}
		
			\begin{itemize}
				\item[(i)]

    Let $S$ be an E-independent set. 
				Assume that the induced subgraph of $S$ does not contain an edge. Since $S$ is an E-independent set, there exists $p \in S$ such that there is $\displaystyle p' \in \langle S\setminus\{p\} \rangle \setminus \bigcup_{a\in S\setminus\{p\}} \langle S\setminus\{a\}\rangle$.
				If there is no $u, v \in S$ such that $uv \in E(G)$, then
				$\langle S\setminus\{p\} \rangle = S \setminus \{p\}$ and 	$\langle S\setminus\{a\}\rangle = S \setminus \{a\}$. This implies $\displaystyle \bigcup_{a\in S\setminus\{p\}} \langle S\setminus\{a\}\rangle = S$, a contradiction. Hence, the induced subgraph of $S$ contains an edge of $G$.
				
				\item[(ii)] Let $S$ be an E-independent set. Assume there exists $u,v,w \in S$ such that $u,v,w$ form $K_3$ in $G$. Since $S$ is an E-independent set, there exists $p \in S$ with $\displaystyle p' \in \langle S\setminus\{p\} \rangle \setminus \bigcup_{a\in S\setminus\{p\}} \langle S\setminus\{a\}\rangle$, for some $p'\in S$. 
		
		Since $u,v,w$ form $K_3$, $\langle S\setminus\{u\}\rangle = \langle S\setminus\{v\}\rangle = \langle S\setminus\{w\}\rangle = \langle S\rangle$, which implies that $\displaystyle \bigcup_{a\in S\setminus\{p\}} \langle S\setminus\{a\}\rangle = \langle S\rangle$. So $\displaystyle\langle S\setminus\{p\}\rangle \subseteq \bigcup_{a\in S\setminus\{p\}} \langle S\setminus\{a\}\rangle$, leading to a contradiction, since $S$ is an E-independent set. Therefore, there does not exist three vertices $u,v,w \in S$ such that $u,v,w$ form a $K_3$ in $G$.
				
				\item[(iii)]  Let $S$ be an E-independent set. 	Suppose, for contradiction, that $S$ contains two vertices $u$ and $v$ such that neither is part of a $K_3$ in $G$. Then, we have $\langle S\setminus\{u\} \rangle= \langle S\rangle \setminus \{u\}$ and $\langle S\setminus\{v\} \rangle= \langle S\rangle \setminus \{v\}$. 
	
	Since $|S| \geq 3$, we consider two cases:
	\begin{description}

		\item[Case 1:] $|S| = 3$.\\
		There exists $w\in S$ such that $u,v \in \langle S\setminus\{w\}\rangle$. Then, we obtain $\langle S\setminus\{u\}\rangle\cup \langle S\setminus\{w\}\rangle=\langle S\setminus\{v\} \rangle\cup \langle S\setminus\{w\}\rangle=\langle S\setminus\{u\} \rangle\cup \langle S\setminus\{v\}\rangle=\langle S\rangle$,  contradicting the assumption that $S$ is an E-independent.
		\item[Case 2:] $|S|>3$.\\
	There exist $w,w'\in S$ such that $u,v \in \langle S\setminus\{w\}\rangle$ and $u,v \in \langle S\setminus\{w'\}\rangle$. This implies, $\langle S\setminus\{u\}\rangle\cup \langle S\setminus\{w\}\rangle=\langle S\setminus\{v\} \rangle\cup \langle S\setminus\{w\}\rangle=\langle S\setminus\{u\} \rangle\cup \langle S\setminus\{w'\}\rangle=\langle S\setminus\{v\} \rangle\cup \langle S\setminus\{w'\}\rangle=\langle S\rangle$, again contradicting the E-independence of  $S$. 	Thus, at most one vertex in $S$ is not part of a  $K_3$.
			
	\end{description}
    \end{itemize}
		\end{proof}

	\begin{proposition}\label{prop-e-independent}
		If $S\subseteq V(G)$ and $u,v,w \in S$ such that $\langle S\setminus\{u\}\rangle\cup \langle S\setminus\{v\}\rangle=\langle S\setminus\{u\}\rangle\cup \langle S\setminus\{w\}\rangle=\langle S\setminus\{v\}\rangle\cup \langle S\setminus\{w\}\rangle=\langle S\rangle$, then $S$ is an E-dependent set.
	\end{proposition}
	\begin{proof}
		For any $x\notin\{u,v,w\}$, we have $\langle S\setminus\{x\}\rangle\subseteq \langle S\setminus\{u\}\rangle\cup \langle S\setminus\{v\}\rangle=\langle S \rangle$.
		Now for $x\in \{u,v,w\}$, for instance take $x=u$. Then $\langle S\setminus\{x\}\rangle\subseteq \langle S\setminus\{v\}\rangle\cup \langle S\setminus\{w\}\rangle  =\langle S \rangle$. That is, for any $p \in S$, $\displaystyle\langle S \setminus \{p\}\rangle\subseteq\bigcup_{a\in S\setminus\{p\}} \langle S \setminus \{a\}\rangle$. Hence, $S$ is an E-dependent set.
	\end{proof}
 Using a similar argument, we get the following remark.
 \begin{remark}\label{remark-e-independent}
     Let $S\subseteq V(G)$ and $u,v,w,x \in S$ such that $\langle S\setminus\{u\}\rangle\cup \langle S\setminus\{v\}\rangle=\langle S\setminus\{w\}\rangle\cup \langle S\setminus\{x\}\rangle=\langle S\rangle$, then $S$ is an E-dependent set.
 \end{remark}

Now we have the following two observations which follow immediately from the definition of $\Delta$-convexity.
	\begin{observation}
		If $G$ is a triangle-free graph, then $e_\Delta(G)=2$.
	\end{observation}

	\begin{observation}
		If $G$ is a complete graph $K_n$ for $n>2$, then $e_\Delta(K_n)=2$.
	\end{observation}
	
From Sierksma's inequalities (Theorem \ref{inequalities}) and Theorem \ref{caratheodory_bound}, we can deduce the following lemma.

\begin{lemma}\label{exchange_bound}
If $G$ is a graph with $k$ triangles, then the exchange number $e_\Delta$ is bounded by $k+2$.
\end{lemma}

This lemma establishes an upper bound on the exchange number based on the number of triangles present in the graph. It is worth noting that there exist graphs that attain the bounds in Lemma \ref{exchange_bound}, as stated in the following remark.

 	\begin{figure}[H]
            \centering
		\epsfxsize=15truecm \slika{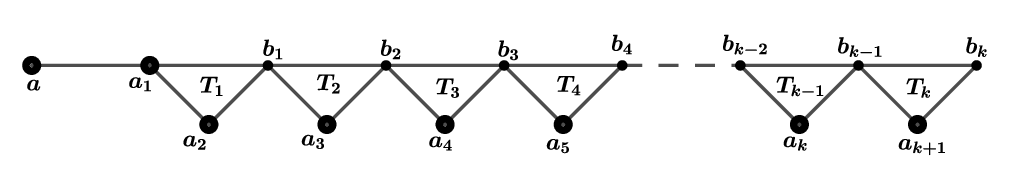}\vspace{-0.4in}
		\caption{Graph $G$ with $k$ triangles and $e_\Delta(G)=k+2$.}
		\label{delta_exchange}
	\end{figure}

	\begin{remark}
		There exists a graph $G$ with $k$ triangles that has an exchange number of $k+2$. For example, in Figure \ref{delta_exchange}, $G$ is a graph with $k$ triangles, and the exchange number is $k+2$. Here, $S=\{a, a_1, a_2, a_3, \ldots, a_{k+1}\}$ is the E-independent set. Since $b_{k} \in \langle S \setminus \{a\} \rangle$ but $\displaystyle b_{k} \notin \bigcup_{b\in S\setminus\{a\}}\langle S\setminus\{b\}\rangle$.
	\end{remark}
Furthermore, the exchange number can be arbitrarily large in the sense that, for any natural number $n$, we can construct a graph $G$ with $e_\Delta(G)=n$, as stated in the following proposition.
	\begin{proposition}
		For any natural number $n>1$, there exists a graph $G$ with $e_\Delta(G)=n$.
	\end{proposition}
	
	\begin{proof}
		For $n=2$, consider any complete graph or triangle-free graph  $G$. In this case, $e_\Delta(G)=2$.
		
		For $n\geq 3$, consider the graph $G$ depicted in Figure \ref{delta}. Let $S=\{a_1,a_2,a_3,\ldots,a_n\}$. Then $\langle S\rangle=V(G)$, $\langle S\setminus\{a_1\}\rangle=\{a_2,a_3,a_4,\ldots, a_n\}$, $\langle S\setminus\{a_2\}\rangle=\{a_1,a_3,a_4,\ldots ,a_n\}$, $\langle S\setminus\{a_i\}\rangle=\{a_1,a_2,b_1,a_3,b_2,a_4,\ldots ,a_{i-1},b_{i-2},a_{i+1},a_{i+2},\ldots ,a_n\}$ for $i\in\{3,4,5,\ldots ,n-1\}$ and $\langle S\setminus\{a_n\}\rangle=\{a_1,a_2,b_1,a_3,b_2,a_4,\ldots,b_{n-3},a_{n-1},b_{n-2}\}$.
		Since $b_{n-2}\in \langle S\setminus\{a_n\}\rangle$ but $ \displaystyle b_{n-2}\notin\bigcup_{a\in S\setminus\{a_n\}}\langle S\setminus\{a\}\rangle$, $S$ is E-independent.
		
		 Now, if we add any element $x \in \{b_1,b_2,b_3,\ldots ,b_{n-2},b\}$ to $S$, then $\langle S\setminus\{a_1\}\rangle=\langle S\setminus\{x\}\rangle=V(G)$. That is, $S$ becomes E-dependent.
		 
		 Next, we need to prove that there is no E-independent set of size greater than $n$. Suppose there is an E-independent set $S'$ with $|S'|\geq n+1$.
		 Consider the possible cases
		 \begin{itemize}
		 	\item[(a)] $S'$ contains all vertices of a triangle, say $\{a,b,c\}$. Then $\langle S'\setminus\{a\}\rangle=\langle S'\setminus\{b\}\rangle=\langle S'\setminus\{c\}\rangle=\langle S'\rangle$, impliying $S'$ is E-dependent.
		 	\item[(b)]$S'$ contains one vertex from each triangle. $G$ contains only $n-1$ triangles. So $S'$ contains two vertices from two triangles, say $T_i$ and $T_j$. Let  $a_i,b_i,a_j,b_j \in S'$, where $a_i,b_i \in T_i$ and $a_j,b_j \in T_j$. Then $\langle S'\setminus{a_i}\rangle=\langle S'\setminus{b_i}\rangle=\langle S'\setminus{a_j}\rangle=\langle S'\setminus{b_j}\rangle=V(G)$. This implies $S'$ is E-dependent. 
		 	\item[(c)] $S'$ does not contain vertices from some triangles. That is, $S'$ contains two vertices from at least three triangles. So there exists a sequence of triangles $T_i,T_{i+1},\ldots, T_{i+r}$ ($i\in \{1,2,\ldots, k\}$ and $i+r\leq k$) such that $S'$ contains at least one vertex from each of these triangles and two vertices from at least two triangles, say $T_x$ and $T_y$, $x,y\in\{i,i+1,\ldots,i+r\}$. Let $a_x,b_x \in T_x$ and $a_y,b_y \in T_y$. Then $\langle S'\setminus{a_x}\rangle=\langle S'\setminus{b_x}\rangle=\langle S'\setminus{a_y}\rangle=\langle S'\setminus{b_y}\rangle=\langle S'\rangle$. This implies $S'$ is E-dependent. 
		 \end{itemize}
		Hence for the graph in Figure~\ref{delta},  $e_\Delta(G)=n$.
	\end{proof}

		%\begin{lemma}\label{any_two_exchange}
	%		If $G$ is a 2-connected chordal graph, then any two vertices form an E-independent set of $G$
		%\end{lemma}
		%\begin{proof}
  	%Let $G$ be a 2-connected chordal graph and $S=\{u,v\}\subseteq V(G)$.
	
	%If $u$ and $v$ are not adjacent, then it is clear that $S$ is an E-independent set.%($\langle S\setminus \{u\}\rangle\nsubseteq \langle S\setminus\{v\}\rangle$ and $\langle S\setminus\{v\}\rangle\nsubseteq \langle S\setminus\{u\}\rangle$).
	
	%If $u$ and $v$ are adjacent, then by Lemma \ref{hull_2_chordal}, $\langle S\rangle=V(G)$. Also, $\langle S\setminus\{u\}\rangle =\{v\}$ and $\langle S\setminus\{v\}\rangle =\{u\}$. This implies $S$ is an E-independent set.
			
	%	\end{proof}

\begin{lemma}
If $G$ is a $2$-connected chordal graph, then the exchange number $e_\Delta(G)$ is either 2 or 3.
\end{lemma}
\begin{proof}
		
	Let $G$ be a $2$-connected chordal graph and $S\subseteq V(G)$ is an E-independent set with $|S|>1$. By Lemma~\ref{adjacency_exchange_indpt}, there exist $a,b\in S$ such that $ab\in E(G)$.
   \begin{description}

       \item[Case 1:] Every vertex of $G$ is adjacent to either $a$ or $b$ or both. In this case, there is no E-independent set of size greater than $2$. Assume that there exist an E-independent set containing a vertex $c$ adjacent to either $a$ or $b$ or both. Without loss of generality, we may assume that $c$ is adjacent to $b$. Then by  Lemma \ref{hull_2_chordal},  $\langle a,b\rangle = \langle b,c\rangle = V(G)$, implies that the set is E-dependent. Therefore the exchange number is $2$.
       \item[Case 2:] There exist vertices not adjacent to $a$ and $b$. Let $c$ be such a vertex. Consider $S'=\{a,b,c\}$. Then $\langle S'\rangle = V(G)$, $\langle S'\setminus\{a\}\rangle = \{b,c\}$, $\langle S'\setminus\{b\}\rangle = \{a,c\}$, and $\langle S'\setminus\{c\}\rangle = V(G)$. Clearly, $S'$ is E-independent.
Now, we need to prove that there is no E-independent set of size greater than $3$. Assume that there is an E-independent set $S''$ with $|S''| > 3$. By Lemma \ref{adjacency_exchange_indpt}, there exist $u,v \in S''$ with $uv \in E(G)$. Since $|S''| > 3$, there exist $x,y \in S''$ other than $u$ and $v$. By Lemma~\ref{hull_2_chordal}, $\langle u,v \rangle = V(G)$. This implies $\langle S''\setminus\{x\}\rangle = \langle S''\setminus\{y\}\rangle = V(G)$, making $S''$ E-dependent. Hence $e_\Delta(G) = 3$.
  \end{description}
\vspace{-.5cm}
   \end{proof}
		%\begin{lemma}
	%		If $G$ is a $2$-connected chordal graph that is not a complete graph, then the exchange number $e_\Delta(G) = 3$.
		%\end{lemma}
		%\begin{proof}
	%		Let $G$ be a $2$-connected chordal graph. Let $S=\{a,b,c\} \subseteq V(G)$, where $ab \in E(G)$ and $\{a,b,c\}$ does not form $K_3$ in $G$. Then $\langle S\rangle = V(G)$, $\langle S\setminus\{a\}\rangle = \{b,c\}$, $\langle S\setminus\{b\}\rangle = \{a,c\}$, and $\langle S\setminus\{c\}\rangle = V(G)$. Clearly, $S$ is E-independent.
			
	%		Now we need to prove that there is no E-independent set of size greater than 3. Assume that there is an E-independent set $S'$ with $|S'| > 3$. By Lemma \ref{adjacency_exchange_indpt}, there exist $u,v \in S'$ with $uv \in E(G)$. Since $|S'| > 3$, there exist $x,y \in S'$ other than $u$ and $v$. By Lemma~\ref{hull_2_chordal}, $\langle u,v \rangle = V(G)$. This implies $\langle S'\setminus\{x\}\rangle = \langle S'\setminus\{y\}\rangle = V(G)$, making $S'$ E-dependent. Hence, $e_\Delta(G) = 3$.
		%\end{proof}
			
For block graphs, the exchange number depends on the number of blocks in the graph. We present the following results for block graphs:	
	\begin{theorem}\label{block_exchange}
			Let $G$ be a block graph with $\ell$ blocks. Then the following statements hold:
	\begin{enumerate}
		\item[(i)] If $G$ contains no blocks isomorphic to the complete graph $K_2$ and all the blocks lie in a single chain, then $e_{\Delta}(G) = \ell + 1$.
		\item[(ii)] If $G$ contains no blocks isomorphic to the complete graph $K_2$ and the blocks lie in more than one chain, and $k$ is the number of blocks in the longest chain in $G$, then $e_{\Delta}(G) = k + 2$.
		\item[(iii)] If $G$ contains blocks isomorphic to the complete graph $K_2$ and $k$ is the maximum number of consecutive non-$K_2$ blocks in a chain, then $e_{\Delta}(G) = k + 2$.
		\end{enumerate}	  
	\end{theorem}
  % \item[(iii)] If $G$ is a block graph and $k$ is the number of blocks in a longest chain in $G$ which does not contains $K_2$ as a block, then $e_{\Delta}(G) = k + 2$.
	\begin{proof}
		\begin{enumerate}
			\item[(i)] Let $G$ be a block graph with $\ell$ number of blocks and no blocks isomorphic to $K_2$ and all blocks lie in a single chain. Let $B_1, B_2, B_3, \ldots, B_{\ell}$ be the blocks of $G$ and let $S$ be an E-independent set of $G$. By Lemma~\ref{adjacency_exchange_indpt} (i) and (ii), $S$ contains two vertices from at least one block and not more than two vertices from a single block. Without loss of generality, we may assume that $S$ contains two vertices from $B_1$ and one vertex each from other blocks. Let $S=\{u_1,{u_1}',u_2, u_3,\ldots,u_\ell\}$, where $u_1,{u_1}'\in B_1$, $u_i \in B_i$ for $i\in\{2,3,\ldots,\ell\}$, and there are no three vertices in $S$ forming $K_3$ in $G$ and the vertices are not cut vertices of $G$. Then we have $\langle S\setminus\{u_1\}\rangle =\{S\setminus u_1\}$, $\langle S\setminus\{{u_1}'\}\rangle =\{S\setminus {u_1}'\}$,   $\langle S\setminus\{u_i\}\rangle = V(B_1)\cup V(B_2)\cup \cdots\cup V(B_{i-1})\cup \{v_{i+1},v_{i+2},\ldots,v_\ell\}$ for $i\in\{2,3,\ldots,\ell\}$. This implies that there exists $v\in V(B_\ell)$ such that $\displaystyle v\in \langle S\setminus\{u_\ell\}\rangle \setminus \bigcup_{a\in S\setminus\{u_\ell\}} \langle S\setminus\{a\}\rangle$. Thus $S$ is an E-independent set of size $\ell+1$ elements.

			Now we have to prove that there is no E-independent set of size greater than $\ell+1$. Suppose there is an E-independent set $S'$ with $|S'| \geq \ell+2$. By Lemma \ref{adjacency_exchange_indpt} (i) and (ii), $S'$ contains two vertices from at least one block and not more than two vertices from a single block.\\
			Consider the following two possible cases:
			\begin{description}
				\item[Case 1:] $S'$ contains at least one vertex from each block. 
    
    Since $|S'|\geq \ell+2$, we have $S'\cap V(B_i)=2$ and $S'\cap V(B_j)=2$ for at least two blocks $B_i$ and $B_j$ with $i, j \in \{1, 2, 3, \ldots, \ell\}$. Let $u, u', v, v' \in S'$, where $u, u' \in V(B_i)$ and $v, v' \in V(B_j)$. Then $\langle S'\setminus\{u\}\rangle = \langle S'\setminus\{v\}\rangle = V(G)$, making $S'$ an E-dependent set.
    
				\item[Case 2:] $S'$ does not contain vertices from some blocks. 
    
Here, $S'$ must contain two vertices, each from at least three blocks. Let $S'$ contain two vertices each of the blocks $B_x,B_y,$ and $B_z$.  Then we can see that $\langle S'\setminus\{u\}\rangle \cup \langle S'\setminus\{v\}\rangle=\langle S'\setminus\{u\}\rangle\cup \langle S'\setminus\{w\}\rangle=\langle S'\setminus\{v\}\rangle\cup \langle S'\setminus\{w\}\rangle= \langle S'\rangle$ for $u\in S'\cap V(B_x), v\in S'\cap V(B_y)$ and $w\in S'\cap V(B_z)$. This implies $S'$ an E-dependent set.  
    
    %There exists a sequence of blocks $B_i, B_{i+1}, \ldots, B_{j}$ such that $S'$ contains at least one vertex from each of these blocks and two vertices from at least two blocks, say $B_x$ and $B_y$ with $x, y \in \{i, i+1, \ldots, j\}$. Let $u, u', v, v' \in S'$, where $u, u' \in B_x$ and $v, v' \in B_y$. Then $\langle S\setminus\{u\}\rangle = \langle S\setminus\{v\}\rangle = \langle S'\rangle$, making $S'$ an E-dependent set.
			\end{description}
			Thus there is no E-independent set of size greater than $\ell+1$. Hence the exchange number $e_{\Delta}(G) = \ell + 1$.
 
\item[(ii)]	Let $G$ be a block graph containing no blocks isomorphic to the complete graph $K_2$ and let the blocks lie in more than one chain with $k$ being the number of blocks in the longest chain. Then from (i), there is an E-independent set of size $k+1$ in the longest chain, let it be $S_k$, and there exists a vertex $v'$ with $ \displaystyle v' \in \langle {S_k}\setminus\{u_n\}\rangle \setminus \bigcup_{a \in {S_k}\setminus\{u_n\}} \langle {S_k}\setminus\{a\}\rangle$, for some $u_n\in S_k$. Let $S = S_k \cup \{v\}$, where $v$ is any vertex in $G$ not in the longest chain of blocks. Also the same vertex $ \displaystyle v' \in \langle {S}\setminus\{v\}\rangle \setminus \bigcup_{a \in S\setminus\{v\}} \langle S\setminus\{a\}\rangle$. Thus $S$ is an E-independent set of size $k+2$.
			
			Now we have to prove there is no E-independent set of size greater than $k+2$. Suppose there is an E-independent set $S'$ with $|S'|\geq k+3$. By Lemma \ref{adjacency_exchange_indpt} (i) and (ii), $S'$ contains two vertices from at least one block and not more than two vertices from any blocks. From (i), it is not possible to get an E-independent set of size greater than $k+1$ from one single chain of blocks. Then $S'$ must contain vertices from at least two chains of blocks.\\
			Consider the possible cases:
			\begin{description}
				\item[Case 1:] $S'$ contains only elements from two distinct chains of blocks.\\
				Let the chains of blocks be $C_1$ and $C_2$ (clearly, $|V(C_1)\cap S'|\geq 2$ and $|V(C_2)\cap S'|\geq 2$ ) with $u\in V(C_1)\cap S'$ and $w,w' \in V(C_2)\cap S'$. In $\langle {S'}\setminus\{u\}\rangle$ contain all elements of $\langle S'\rangle$ except some elements of $\langle S'\rangle \cap V(C_1)$. Similarly, $\langle {S'}\setminus\{w\}\rangle$ and $\langle {S'}\setminus\{w'\}\rangle$ contain all elements of $\langle S'\rangle$ except some elements of $\langle S'\rangle \cap V(C_2)$. Then $ \displaystyle \langle {S'}\setminus\{u\}\rangle \cup \langle {S'}\setminus\{w\}\rangle = \langle {S'}\setminus\{u'\}\rangle \cup \langle {S'}\setminus\{w'\}\rangle = \langle S'\rangle$. Then by Remark~\ref{remark-e-independent}, $S'$ is an E-dependent set.
				\item[Case 2:] $S'$  contains elements from more than two  distinct chains of blocks.\\
				Let $C_1$, $C_2$ and $C_3$ be  any three  distinct chains of blocks with $u\in V(C_1)\cap S'$, $w \in V(C_2)\cap S'$ and $w'\in V(C_3)\cap S'$. In $\langle {S'}\setminus\{u\}\rangle$, contains all elements of $\langle S'\rangle$ except some  elements of $\langle S'\rangle \cap V(C_1)$. Similarly, $\langle {S'}\setminus\{w\}\rangle$ contains all elements of $\langle S'\rangle$ except some  elements of $\langle S'\rangle \cap V(C_2)$, and $\langle {S'}\setminus\{w'\}\rangle$ contains all elements of $\langle S'\rangle$ except some or all elements of $\langle S'\rangle \cap V(C_3)$. Then $\displaystyle \langle {S'}\setminus\{u\}\rangle \cup \langle {S'}\setminus\{w\}\rangle = \langle {S'}\setminus\{u\}\rangle \cup \langle {S'}\setminus\{w'\}\rangle = \langle {S'}\setminus\{w\}\rangle \cup \langle {S'}\setminus\{w'\}\rangle = \langle S'\rangle$. Then by Proposition~\ref{prop-e-independent}, $S'$ is an E-dependent set.
			\end{description}
			Thus, there is no E-independent set of size greater than $k+2$. Hence, the exchange number $e_{\Delta}(G) = k+2$.
				
				\item[(iii)]  Let $G$ be a block graph containing blocks isomorphic to the complete graph $K_2$, and let $k$ be the maximum number of consecutive blocks that does not contain $K_2$ 
 as a block among the chains of blocks in $G$. Then from (i), there is an E-independent set of size $k+1$ from these consecutive blocks, say $S_k$. Then,  there exists $v'$ with $\displaystyle v' \in \langle {S_k}\setminus\{u_n\}\rangle \setminus \bigcup_{a \in {S_k}\setminus\{u_n\}} \langle {S_k}\setminus\{a\}\rangle$. Now, let $S=S_k\cup\{v\}$ where $v$ is any vertex in $G$ not in the consecutive blocks. Also, the same vertex $\displaystyle v' \in \langle {S}\setminus\{v\}\rangle \setminus \bigcup_{a \in S\setminus\{v\}} \langle S\setminus\{a\}\rangle$. Thus, $S$ is an E-independent set of size $k+2$.
			
			Now, we need to prove that there is no E-independent set of size greater than $k+2$. Suppose there is an E-independent set $S'$ with $|S'|\geq k+3$. By Lemma \ref{adjacency_exchange_indpt} (i) and (ii), $S'$ contains two vertices from at least one block and not more than two vertices from any blocks. Since it is not possible to obtain an E-independent set of size $k+3$ from consecutive blocks of a single chain, $S'$ contains at most $k+1$ elements from consecutive blocks of a single chain.
			%By Lemma~\ref{adjacency_exchange_indpt} (iii), $S'$ contains at most one vertex that is not part of any $K_3$.
   So $S'$ contains elements from non-consecutive blocks. From Lemma \ref{adjacency_exchange_indpt} (i), there exist at least two vertices say, $u,v\in S'$ with $uv\in E(G)$, let  $C$ be the consecutive block of chain containing $u$ and $v$. Let $w$ and $x$ be in $S'$ but not in $C$. Then, both $\langle S\setminus\{u\}\rangle$ and $\langle S\setminus\{v\}\rangle$ contain all elements of $\langle S'\rangle$ except some elements of $V(C)\cap \langle S'\rangle$. Also, $\langle S\setminus\{w\}\rangle$ and $\langle S\setminus\{x\}\rangle$ contain  all the elements in $V(C)\cap \langle S'\rangle$. Therefore, $\langle {S'}\setminus\{u\}\rangle \cup \langle {S'}\setminus\{w\}\rangle = \langle {S'}\setminus\{v\}\rangle \cup \langle {S'}\setminus\{x\}\rangle = \langle S'\rangle$. Then by Remark~\ref{remark-e-independent}, $S'$ is an E-dependent set.
			
			Thus, there is no E-independent set of size greater than $k+2$. Hence, the exchange number $e_{\Delta}(G) = k+2$.
			\end{enumerate}
		\end{proof}

%%%%%%%%%%%%%%%%%%%%%%%%%%%%%%%%%%%%%%%%%%%%%%

%%%%%%%%%%%%%%%%%%%%%%%%%%%%%%%%%%%%%%%%%%%%%%%%%%

\section{The Carath\'{e}odory and Exchange Number in Graph Products}\label{section-product}

 In this section, we give a lower bound for exchange number and Carath\'{e}odory number for $\Delta$-convexity in Cartesian product and exact value for strong and lexicographic product of graphs. 

\subsection{Exchange number}
If a graph contains no triangles, then the only E-independent sets are the two element sets. In this case exchange number is $2$. If $G$ and $H$ does not contain $K_3$, then by the definition of the Cartesian product of graphs $G\Box H$, will not contains $K_3$ and the $E$-independent sets in $G\Box H$ are two-element sets. Therefore $e_{\Delta}(G\Box H)=2$.

\begin{theorem}\label{cartexchange}
 Let $G$ and $H$ be two nontrivial connected graphs with $e_{\Delta}(G), e_{\Delta}(H)>2$, then \\ $e_{\Delta}(G\Box H)\geq (e_{\Delta}(G)-1)(e_{\Delta}(H)-1)+1$.

\end{theorem}
\begin{proof}

Let $S_1$ and $S_2$ be E-independent sets in $G$ and $H$ respectively with $|S_1|=e_{\Delta}(G)$ and $|S_2|=e_{\Delta}(H)$. Let $g$ and $h$ be the pivots of $S_1$ and $S_2$ respectively. Then there exists some $x\in \langle S_1 \setminus  \{g\}\rangle$, but $\displaystyle x\notin \bigcup_{a\in S_1 \setminus  \{g\}}\langle S_1 \setminus  \{a\}\rangle $ and some $y\in \langle S_2\setminus \{h\}\rangle$, but $\displaystyle y\notin \bigcup_{b\in S_2 \setminus  \{h\}}\langle S_2 \setminus  \{b\}\rangle $. Let $S=(S_1\setminus \{g\})\times (S_2\setminus \{h\})\cup \{(g,h)\}$. Then $|S|=(|S_1|-1)(|S_2|-1)+1=(e_{\Delta}(G)-1)(e_{\Delta}(H)-1)+1$. 

\noindent {\bf Claim}: $(g,h)$ is the pivot of $S$.

Since $x\in \langle S_1\setminus \{g\}$, for any $h'\in S_2\setminus \{h\}$, $(x,h')\in \langle (S_1\setminus \{g\})\times h'\rangle$. Therefore $\{x\}\times (S_2\setminus \{h\})\subseteq \langle (S_1\setminus \{g\})\times (S_2\setminus \{h\})\rangle\subseteq \langle S\rangle$.

Since $y\in \langle S_2\setminus \{h\}\rangle$, $(x,y)\in \langle \{x\}\times (S_2\setminus \{h\})\rangle\subseteq \langle S\rangle$. Let $(a,b)\in \langle S\setminus \{(g,h)\}\rangle$. Then $(a,y)\notin \langle \{a\}\times (S_2\setminus \{b\})\rangle$ or $(x,b)\notin \langle (S_1\setminus \{a\})\times \{b\} \rangle$ or both. Then $(x,y)\notin \langle S\setminus \{(a,b)\}\rangle $. i.e., $\displaystyle(x,y)\notin \bigcup_{(a,b)\in S\setminus \{(g,h)\}}\langle S\setminus \{(a,b)\}\rangle $. Therefore $S$ is an E-independent set in $G\Box H$ having cardinality $(e_{\Delta}(G)-1)(e_{\Delta}(H)-1)+1$.

\end{proof}
The following theorem gives the sharpness of the lower bound of exchange number in $G\Box H$.

\begin{theorem}
    Let $G$ be a nontrivial connected graph having $e_{\Delta}(G)\geq 3$ and $P_n$ be a path of order $n$, then $e_{\Delta}(G\Box P_n)=e_{\Delta}(G)$.
\end{theorem}
\begin{proof}
    By Theorem~\ref{cartexchange}, $e_{\Delta}(G\Box P_n)\geq (e_{\Delta}(G)-1)(e_{\Delta}(P_n)-1)+1= (e_{\Delta}(G)-1)(2-1)+1=e_{\Delta}(G)$. Let $S\subset V(G\Box P_n)$ with $|S|=e_{\Delta}(G)+1$. Here arise the following possibilities.\\
    
    \noindent
{\bf Case 1:} $S\subseteq V(G^h)$ for some $h\in V(P_n)$.

    Since $S$ has cardinality $e_{\Delta}(G)+1$ and $S\subset V(G^h)$, $|p_G(S)|=e_{\Delta}(G)+1$ and is an $E$-dependent set in $G$. i.e., $\displaystyle\langle p_G(S)\setminus \{a\}\rangle\subseteq \bigcup_{x\in p_G(S)\setminus \{a\}}\langle p_G(S)\setminus\{x\}\rangle $ for any $a\in p_G(S)$. Then by the definition of Cartesian product of graphs, $\displaystyle\langle S\setminus \{(a,b)\}\rangle\subseteq \bigcup_{(x,y)\in S\setminus \{(a,b)\}}\langle S\setminus\{(x,y)\}\rangle $ for any $(a,b)\in S$. Therefore $S$ is an $E$-dependent set in $G\Box P_n$.\\
    
\noindent
    {\bf Case 2:} $S$ intersects at least two $G$ layers.

Assume $S$ intersects with $V(G^{h_1}), V(G^{h_2}), V(G^{h_3}),\ldots , V(G^{h_r})$, for some $h_1,h_2,h_3,\ldots, h_r\in V(P_n)$.
    Since $P_n$ does not contains $K_3$, any $U\subseteq P_n$, $\langle U\rangle=U$. Then $\langle S\rangle =\langle S\cap V(G^{h_1})\rangle \cup \langle S\cap V(G^{h_2})\rangle \cup \langle S\cap V(G^{h_3})\rangle\cup\ldots \cup\langle S\cap V(G^{h_r})\rangle$. 
    
    Let $(g,h_i)\in S$, $i\in \{1,2,\ldots,r\}$. For $j\neq i$ ($j=1,2,\ldots,r$) there is some $g'\in V(G)$ with $(g',h_j)\in S$. Then $\displaystyle \langle S\setminus \{(g',h_j)\}\rangle =\bigcup_{k\neq j}\langle S\cap G^{h_k}\rangle\cup \langle S\cap V(G^{h_j})\setminus \{(g',h_j)\}\rangle$. Since $e_{\Delta}(G)\geq 3$, $S$ contains at least four vertices. So we can find some $(g',h')\in S$ different from $(g,h_i)$ and $(g',h_j)$. \\
    
\noindent
{\bf Subcase 1:} $r=2$.

In this case $S$ intersects with only two $G$ layers. If $S$ contains at least two vertices each from the intersecting two $G$ layers, then $\langle S\setminus \{(g',h')\}\rangle$ contains $\langle S\cap V(G^{h_j})\rangle$. If one intersecting $G$ layer contains only one vertex of $S$, then the other $G$ layer contains $e_{\Delta}(G)$ number of vertices. Now also $\langle S\setminus \{(g',h')\}\rangle$ contains $\langle S\cap V(G^{h_j})\rangle$.\\
    
\noindent
{\bf Subcase 2:} $r\geq 2$.  

In this case we can find some $(g',h')\in S$ with $h'\neq h_j$. Then $\langle S\setminus \{(g',h')\}\rangle$ contains $\langle S\cap V(G^{h_j})\rangle$.
    
From the above cases, we can conclude that for any $h_i\in V(P_n), i\in \{1,2,\ldots,r\}$, $\displaystyle \langle S\setminus \{(g,h_i)\}\rangle\subseteq \bigcup_{(a,b)\in S\setminus \{(g,h)\}}\langle S\setminus \{(a,b)\}\rangle$. Therefore, $S$ is an $E$-dependent set in $G\Box P_n$.  Hence $e_{\Delta}(G\Box P_n)=e_{\Delta}(G)$.
\end{proof}
\begin{theorem}\label{exstrong}
	Let $G$ and $H$ be two nontrivial connected graphs with at least one of them has diameter greater than two, then $e(G\boxtimes H)= 3$.
\end{theorem}
\proof Assume that $G$ has diameter greater than or equal to two. Then there exist an induced path $g_1g_2g_3$ in $G$. Let $h_1h_2$ be any edge in $H$. Our aim is to prove the set $S=\{(g_1,h_1),(g_1,h_2),(g_3,h_1)\}$ is E-independent in $G\boxtimes H$.\\

\noindent
\textbf{Claim} $(g_3,h_1)$ is the pivot of $S$.

$\langle S \setminus  (g_3,h_1) \rangle=\langle\{(g_1,h_1),(g_1,h_2)\}\rangle=V(G\boxtimes H)$. Because the first iteration contains $N_G(g_1)\times \{h_1,h_2\}$ and $\{g_1\}\times (N_H(h_1)\cap N_H(h_2))$. In the second iteration contains $N_G[g_1]\times (N_H(h_2) \setminus  \{h_1\})$ and $N_G[g_1]\times (N_H(h_1) \setminus  \{h_2\})$. Continue this process until it cover all the vertices of $G\boxtimes H$. Since $g_1$ and $g_3$ are not adjacent in $G$ and the definition of strong product of graphs,  $(g_1,h_1)$  and $(g_3,h_1)$ are not adjacent in $G\boxtimes H$. Similarly $(g_1,h_2)$ and $(g_3,h_1)$ are not adjacent in $G\boxtimes H$. Then $\langle\{(g_1,h_1),(g_3,h_1)\}\rangle=\{(g_1,h_1),(g_3,h_1)\}$ and $\langle\{(g_1,h_2),(g_3,h_1)\}\rangle=\{(g_1,h_2),(g_3,h_1)\}$. Therefore $(g_3,h_1)$ is the pivot of $S$. Hence $S$ is E-independent. So $e(G\boxtimes H)\geq 3$.

Now we have to prove for any four element subset of $V(G\boxtimes H)$ is E-dependent. \\Let $\{(g_1,h_1),(g_2,h_2),(g_3,h_3),(g_4,h_4)\}\subseteq V(G\boxtimes H)$. If all the four vertices are not mutually adjacent, then it is E-dependent. Assume $(g_1,h_1)(g_2,h_2)\in E(G\boxtimes H)$. From the above case, we can say that the convex hull  $\langle \{(g_1,h_1),(g_2,h_2),(g_3,h_3)\}\rangle=V(G\boxtimes H)$ and $\langle\{(g_1,h_1),(g_2,h_2),(g_4,h_4)\}\rangle=V(G\boxtimes H)$. Hence $\{(g_1,h_1),(g_2,h_2),(g_3,h_3),(g_4,h_4)\}$ is E-dependant. Therefore $e(G\boxtimes H)= 3$.\qed

\begin{definition}
	A graph $G$ is said to have the edge-vertex property if there exists an edge $uv$ in $E(G)$ and a vertex $x\in V(G)$, such that $d(u,x)\geq 2$ and $d(v,x)\geq 2$
\end{definition}
\begin{theorem}\label{exlex}
	Let $G$ and $H$ be two nontrivial connected graphs. Then \\ $e_{\Delta}(G\circ H)= \left\{
	\begin{array}{ll}
	
	3, & \hbox{if $G$ has diameter at least two or $H$ has the edge-vertex property;} \\
	2, & \hbox{otherwise.}
	\end{array}
	\right.
	$
\end{theorem}
\begin{proof}

Assume $G$ and $H$ be two nontrivial connected graphs $G$ has diameter less than two and $H$ does not have the edge-vertex property. Then we have to prove that for any $3$ element subset of $V(G\circ H)$ is E-dependent. Let $\{(g_1,h_1),(g_2,h_2),(g_3,h_3)\}\subseteq V(G\circ H)$.\\

\noindent
\textbf{Case 1:} $g_1,g_2,g_3$ are pairwise distinct vertices.\\
In this case the vertices $(g_1,h_1),(g_2,h_2)$ and $(g_3,h_3)$ are in the different $H$-layers. If $G$ has diameter less than 2, then $G$ is a complete graph. Then the induced subgraph of $\{(g_1,h_1),(g_2,h_2),(g_3,h_3)\}$ is $K_3$. Then $\langle(\{(g_1,h_1),(g_2,h_2),(g_3,h_3)\}) \setminus  \{(g_i,h_i)\}\rangle=V(G\circ H)$, for $i=1,2,3$, hence $\{(g_1,h_1),(g_2,h_2),(g_3,h_3)\}$ is an E-dependent set.\\

\noindent
\textbf{Case 2:} $g_1=g_2\neq g_3$.\\
Let $g_1=g_2=g$. $\langle\{(g,h_1),(g,h_2)\}\rangle= V(G\circ H)$, since the vertices in $N_G(g)\times V(H)$ are adjacent to both $(g,h_1)$ and $(g,h_2)$. Then  $N_G(g)\times V(H)\subseteq \langle\{(g,h_1),(g,h_2)\}\rangle$. Now we can do the same process for $N_G(g)$ vertices and continue this process to get all the vertices of $G\circ H$. Therefore $\langle\{(g,h_1),(g,h_2)\}\rangle=V(G\circ H)$. Since $G$ is complete, in the first iteration of the convex hull of $\{(g,h_1),(g_3,h_3)\}$ contains $V(G) \setminus  \{g,g_3\} \times V(H)$ and $\{g,g_3\} \times V(H)$. Therefore $\langle\{(g,h_1),(g_3,h_3)\}\rangle=V(G\circ H)$. Similarly we get $\langle\{(g,h_2),(g_3,h_3)\}\rangle=V(G\circ H)$. Hence $\{(g,h_1),(g,h_2),(g_3,h_3)\}$ is E-dependent.\\

\noindent
\textbf{Case 3:} $g_1=g_2= g_3$.\\
If $h_1,h_2,h_3$ are mutually non adjacent vertices, then $\{(g,h_1),(g,h_2),(g,h_3)\}$ is E-dependent. Assume $h_1h_2\in E(H)$, then either $h_3$ is adjacent to $h_1$ or $h_3$ because $H$ does not have the edge-vertex property. Assume that $h_1h_2h_3$ be a path in $H$. Then as in the Case 2 above $\langle \{(g,h_1),(g,h_2),(g,h_3)\} \setminus  \{(g,h_i)\}\rangle=V(G\circ H)$, for $i=1,2$. Hence $\{(g,h_1),(g,h_2),(g,h_3)\}$ is E-dependent.

Hence $e(G\circ H)=2$, when $G$ has diameter less than two and $H$ does not have the edge-vertex property.

Assume that $G$ has diameter at least two or $H$ has an edge-vertex property. If $H$ has an edge-vertex property. Then there exists an edge $h_1h_2 \in E(H)$ and $h_3\in V(H)$ such that $d(h_1,h_3)\geq 2$ and $d(h_2,h_3)\geq 2$. Then for any $g\in V(G)$, $\{(g,h_1),(g,h_2),(g,h_3)\}$ is E-independent. As in the case2 above,  $\langle\{(g,h_1),(g,h_2)\}\rangle= V(G\circ H)$. But $\langle \{(g,h_1),(g,h_3)\}\rangle=\{(g,h_1),(g,h_3)\}$ and $\langle \{(g,h_2),(g,h_3)\}\rangle=\{(g,h_2),(g,h_3)\}$. Hence $\{(g,h_1),(g,h_2),(g,h_3)\}$ is E-independent in $G\circ H$.

Suppose  $G$ has diameter at least two. If $H$ has an edge-vertex property, then we can find  a 3-element E-independent set. So assume $H$ does not have the edge-vertex property. Since $G$ has diameter at least two, we can find a path $g_1g_2g_3$ in $G$. Let $h_1h_2\in E(H)$. Consider the set $\{(g_1,h_1)(g_1,h_2)(g_3,h_1)\}$. Then $\langle\{(g_1,h_1),(g_1,h_2)\}\rangle= V(G\circ H)$, but $\langle\{(g_1,h_1),(g_3,h_1)\}\rangle= \{(g_1,h_1),(g_3,h_1)\}$ and $\langle\{(g_1,h_2),(g_3,h_1)\}\rangle= \{(g_1,h_2),(g_3,h_1)\}$. Hence $S$ is E-independent.

We have to prove any four element set in $G\circ H$ is E-dependent. \\ Let $S=\{(g_1,h_1),(g_2,h_2),(g_3,h_3),(g_4,h_4)\}\subseteq V(G\circ H)$.
If $g_i\neq g_j$ and $g_ig_j\notin E(G)$ for i,j=1,2,3,4, then $S$ is E-dependent.
Since $G\boxtimes H\subseteq G\circ H$, if $g_ig_j\in E(G)$, for $i,j=1,2,3,4$ or $h_ih_j\in E(H)$, then $S$ is E-independent.
\end{proof}
\subsection{ Carath\'{e}odory Number}

\begin{theorem}\label{cartcara}
 Let $G$ and $H$ be two nontrivial connected graphs with $c_{\Delta}(G), c_{\Delta}(H)>2$, then \\ $c_{\Delta}(G\Box H)\geq c_{\Delta}(G)c_{\Delta}(H)$.

\end{theorem}
\begin{proof}

Let $S_1$ and $S_2$ be C-independent sets in $G$ and $H$ respectively with $|S_1|=c_{\Delta}(G)$ and $|S_2|=c_{\Delta}(H)$. Then there exists some $x\in \langle S_1 \rangle$, but $\displaystyle x\notin \bigcup_{a\in S_1}\langle S_1 \setminus  \{a\}\rangle $ and some $y\in \langle S_2 \rangle$, but $\displaystyle y\notin \bigcup_{b\in S_2} \langle S_2 \setminus  \{b\}\rangle $. Let $S=S_1\times S_2$. Then $|S|= c_{\Delta}(G)c_{\Delta}(H)$. 

\noindent {\bf Claim}: $(x,y)\in \langle S\rangle$, but $\displaystyle(x,y)\notin \bigcup_{(a,b)\in S} \langle S\setminus \{(a,b)\}\rangle $.

Since $x\in \langle S_1\rangle$, for any $h\in S_2$, $(x,h)\in \langle S_1\times h\rangle$. Therefore $\{x\}\times S_2\subseteq \langle S_1\times S_2\rangle= \langle S\rangle$.
Since $y\in \langle S_2 \rangle$, $(x,y)\in \langle \{x\}\times S_2\rangle\subseteq \langle S\rangle$. Let $(a,b)\in \langle S\rangle$. Then $(a,y)\notin \langle \{a\}\times S_2\setminus \{b\}\rangle$ or $(x,b)\notin \langle S_1\setminus \{a\}\times \{b\}\times \rangle$ or both. Then $(x,y)\notin \langle S\setminus \{(a,b)\}\rangle $. i.e., $\displaystyle(x,y)\notin \bigcup_{(a,b)\in S\setminus \{(g,h)\}}\langle S\setminus \{(a,b)\}\rangle $. Therefore $S$ is C-independent set in $G\Box H$ having cardinality $c_{\Delta}(G)c_{\Delta}(H)$. Hence $c_{\Delta}(G\Box H)\geq c_{\Delta}(G)c_{\Delta}(H)$.
Hence $c_{\Delta}(G\Box P_n)=c_{\Delta}(G)$.
\end{proof}
The following theorem gives the sharpness of the lower bound of the Carathéodory number in $G\Box H$.

\begin{theorem}
    Let $G$ be a nontrivial connected graph having $c_{\Delta}(G)\geq 4$ and $P_n$ be a path of order $n$, then $c_{\Delta}(G\Box P_n)=c_{\Delta}(G)$.
\end{theorem}
\begin{proof}
    By Theorem~\ref{cartcara} $c_{\Delta}(G\Box P_n)\geq c_{\Delta}(G)c_{\Delta}(P_n)=c_{\Delta}(G)$. Let $S\subset V(G\Box P_n)$ with $|S|=c_{\Delta}(G)+1$. Here arise the following possibilities.\\

\noindent
    {\bf Case 1:} $S\subseteq V(G^h)$ for some $h\in V(P_n)$.

    Since $S$ has cardinality $e_{\Delta}(G)+1$ and $S\subset V(G^h)$, $|p_G(S)|=c_{\Delta}(G)+1$ and is an C-dependent set in $G$. i.e., $\displaystyle\langle p_G(S)\rangle\subseteq \bigcup_{x\in p_G(S)}\langle p_G(S)\setminus \{a\}\rangle $. Again since $S\subseteq V(G^h)$, $\displaystyle\langle S\rangle\subseteq \bigcup_{(x,y)\in S}\langle S\setminus\{(x,y)\}\rangle $. Therefore $S$ is C-dependent in $G\Box P_n$.\\

\noindent
{\bf Case 2:} $S$ intersects at least two $G$ layers.

Assume $S$ intersects with $V(G^{h_1}), V(G^{h_2}), V(G^{h_3}),\ldots , V(G^{h_r})$, for some $h_1,h_2,h_3,\ldots, h_r\in V(P_n)$.
    Since $P_n$ does not contains $K_3$, any $U\subseteq P_n$, $\langle U\rangle=U$. Then $\langle S\rangle =\langle S\cap V(G^{h_1})\rangle \cup \langle S\cap V(G^{h_2})\rangle \cup \langle S\cap V(G^{h_3})\rangle\cup\ldots \cup\langle S\cap V(G^{h_r})\rangle$. 
    
    Let $(g,h_i)\in S$, $i\in \{1,2,\ldots,r\}$. For $j\neq i$ ($j=1,2,\ldots,r$) there is some $g'\in V(G)$ with $(g',h_j)\in S$. Then $\langle S\setminus \{(g',h_j)\}\rangle =\bigcup_{k\neq j}\langle S\cap G^{h_k}\rangle\cup \langle S\cap V(G^{h_j})\setminus \{(g',h_j)\}\rangle$. Since $c_{\Delta}(G)\geq 4$, $S$ contains at least four vertices. So we can find some $(g',h')\in S$ different from $(g,h_i)$ and $(g',h_j)$.\\

\noindent 
{\bf Subcase 1:} $r=2$.

In this case $S$ intersects with only two $G$ layers. If $S$ contains at least two vertices each from the intersecting two $G$ layers, then $\langle S\setminus \{(g',h')\}\rangle$ contains $\langle S\cap V(G^{h_j})\rangle$. If one intersecting $G$ layer contains only one vertex of $S$, then the other $G$ layer contains $c_{\Delta}(G)$ number of vertices. Now also $\langle S\setminus \{(g',h')\}\rangle$ contains $\langle S\cap V(G^{h_j})\rangle$.\\

\noindent
{\bf Subcase 2:} $r\geq 2$

    In this case we can find some $(g',h')\in S$ with $h'\neq h_j$. Then $\langle S\setminus \{(g',h')\}\rangle$ contains $\langle S\cap V(G^{h_j})\rangle$
    
    From the above cases, we can conclude that for any $h_i\in V(P_n), i\in \{1,2,\ldots,r\}$, $\displaystyle \langle S\rangle\subseteq \bigcup_{(a,b)\in S\setminus \{(g,h)\}}\langle S\setminus \{(a,b)\}\rangle$. Therefore, $S$ is an $E$-dependent set in $G\Box P_n$.\\  Hence $c_{\Delta}(G\Box P_n)=c_{\Delta}(G)$.
\end{proof}

We conclude this article with the following remark.
\begin{remark}
    Let $G$ and $H$ be two nontrivial connected graphs with $xy\in E(G\ast H)$, for $\ast\in\{\boxtimes, \circ\}$, then from the proofs of Theorem \ref{exstrong} and Theorem \ref{exlex} we get $\langle\{x,y\} \rangle=V(G\ast H)$, for $\ast\in \{\boxtimes, \circ\}$. Therefore $c_\Delta(G\ast H)= 2$, for $\ast\in \{\boxtimes, \circ\}$.
\end{remark}
\subsubsection*{Acknowledgment:} Arun Anil acknowledges the financial support from the University of Kerala, for providing the University Post Doctoral Fellowship (Ac EVII 5911/2024/UOK dated 18/07/2024). 
\bibliographystyle{amsplain}
\bibliography{carathbib}

\end{document}